\documentclass[12pt]{amsart}

\usepackage[colorlinks,citecolor=blue]{hyperref} 
\usepackage{cleveref}
\usepackage{todonotes}
\usepackage{standalone}

\theoremstyle{plain}
\numberwithin{equation}{section} 
\numberwithin{figure}{section}
\newtheorem{theorem}{Theorem}[section]
\newtheorem{corollary}[theorem]{Corollary}
\newtheorem{lemma}[theorem]{Lemma}
\newtheorem{proposition}[theorem]{Proposition}
\newtheorem{question}[theorem]{Question}

\crefformat{equation}{(#2#1#3)}
\crefmultiformat{equation}{(#2#1#3)}{ and~(#2#1#3)}{, (#2#1#3)}{ and~(#2#1#3)}

\newcommand{\calF}{\mathcal{F}}
\newcommand{\QQ}{\mathbb{Q}}
\newcommand{\ZZ}{\mathbb{Z}}
\newcommand{\RR}{\mathbb{R}}
\newcommand{\im}{\operatorname{im}}

\begin{document}

\author{Bal\'azs Strenner}
\address{School of Mathematics, Georgia Institute of Technology,
  Atlanta, GA 30332}
\email{strenner@math.gatech.edu}

\date{\today} \title[Pseudo-Anosov maps with vanishing SAF]{Lifts of
  pseudo-Anosov homeomorphisms of nonorientable surfaces have
  vanishing SAF invariant}

\begin{abstract}
  We show that any pseudo-Anosov map that is a lift of a pseudo-Anosov
  homeomorphism of a nonorientable surface has vanishing SAF
  invariant. We also provide a criterion to certify that a
  pseudo-Anosov map is not such a lift.
\end{abstract}

\maketitle

\section{Introduction}

An \emph{interval exchange transformation} is a bijection
$f: [0,a) \to [0,a)$ with the property that there exist
$0 = a_0 < a_1 < \cdots < a_{n} = a$ and $t_1,\ldots,t_{n} \in \RR$
such that $f|_{[a_{i-1},a_{i})}(x) = x+t_i$ for $1\le i \le n$. The
\emph{Sah-Arnoux-Fathi (SAF) invariant} of $f$ takes values in
$\RR \wedge_\QQ \RR$, and it is defined as
\begin{displaymath}
  SAF(f) = \sum_{i=1}^n (a_{i}-a_{i-1}) \wedge_\QQ t_i.
\end{displaymath}

Given a transversely orientable singular measured foliation $\calF$ on
an orientable surface, any transverse arc gives rise to an interval
exchange transformation $f$ by the first return map of the flow along
the leaves of $\calF$. It turns out that $SAF(f)$ is independent of
the choice of the arc, hence one obtains the notion of the SAF
invariant for foliations by setting $SAF(\calF) = SAF(f)$. Note that
if $\calF$ is not transversely orientable or the surface is
nonorientable, then the first return map is not an interval exchange
transformation in the above sense, hence (at least in this way)
$SAF(\calF)$ cannot be not defined.

A homeomorphism $\psi$ of a surface is \emph{pseudo-Anosov} if there
is a number $\lambda>1$, and a pair of transverse invariant singular
measured foliations $\calF^u$ and $\calF^s$ such that
$\psi(\calF^u) = \lambda\calF^u$ and $\psi(\calF^s) =
  \lambda^{-1}\calF^s.$
The number $\lambda$ is called the stretch factor (or dilatation) of
$\psi$. \cite{FLP}

If $\calF^u$ and $\calF^s$ are transversely orientable (which implies
the orientability of the surface), then $SAF(\calF^u)$ and
$SAF(\calF^s)$ are defined up to scale. It turns out that
$SAF(\calF^u) = 0$ if and only if $SAF(\calF^s) = 0$ \cite[Lemma
2]{CaltaSchmidt13}, and in this case we say that $\psi$ has vanishing
SAF invariant.

To make the SAF invariant well-defined, in the rest of the paper we
assume without mentioning that $\calF^u$ and $\calF^s$ are
transversely orientable when $\psi$ is supported on an orientable
surface. On a nonorientable surface, they cannot both be transversely
oriented, so we only assert that one of them is. This ensures that the
surface has a double cover where both foliations are transversely
orientable.

It is well-understood when the SAF invariant of foliations vanish in
genus 2, and this was used for classifying Teichm\"uller curves in
genus 2 by Calta \cite{Calta04} and McMullen \cite{McMullen03}. In
higher genera, the picture is more complicated. Pseudo-Anosov maps
with vanishing SAF invariant (which do not exist in genus 2) have been
constructed
\cite{ArnouxYoccoz81,ArnouxSchmidt09,CaltaSchmidt13,DoSchmidt16} in
genus 3 and up.

We say that a pseudo-Anosov homeomorphism $\tilde\psi$ of an orientable
surface $S$ is a \emph{nonorientable lift} if
\begin{itemize}
\item $S$ is the oriented double cover of a nonorientable surface $N$,
\item and there is a pseudo-Anosov homeomorphism $\psi$ of $N$ such
  that $\tilde\psi$ is a lift of $\psi$.
\end{itemize}

The orientable double cover is unique: it corresponds to the index two
subgroup of orientation-preserving loops in $\pi_1(N)$. The lift of a
pseudo-Anosov homeomorphism is a pseudo-Anosov homeomorphism with same
stretch factor, and the invariant foliations upstairs are the lifts of
the foliations downstairs.

\begin{theorem}\label{theorem:SAF}
  If a pseudo-Anosov map is a nonorientable lift, then it has
  vanishing SAF invariant.
\end{theorem}

This yields a new large collection of pseudo-Anosov maps with
vanishing SAF invariant. It also explains why the Arnoux--Yoccoz
\cite{ArnouxYoccoz81} and Arnoux--Rauzy (\cite[Section
4.1]{DoSchmidt16}, \cite{ArnouxRauzy91}) examples have vanishing SAF invariant. We were
unable to find a reference for the fact that these examples are
nonorientable lifts, so we elaborate on this in \Cref{sec:AY-and-AR}.

The construction of the other known examples with vanishing SAF does
not involve nonorientable surfaces, so the following question arises.

\begin{question}
  Which of the known examples of pseudo-Anosov maps with vanishing SAF
  invariant are nonorientable lifts?
\end{question}

\begin{theorem}\label{theorem:criterion}
  Let $\tilde\psi$ be a pseudo-Anosov homeomorphism of the closed orientable
  surface of genus $g$. Suppose $\tilde\psi$ is a nonorientable lift. 

  Then the stretch factor $\lambda$ is a root of a monic polynomial
  $p(x) \in \ZZ[x]$ of degree $g$ whose constant coefficient is
  $\pm 1$. Moreover, $p(x)$ is reciprocal mod 2.
\end{theorem}

\begin{corollary}\label{cor:not_lifts}
  Pseudo-Anosov maps with vanishing SAF invariant that are not
  nonorientable lifts include:
  \begin{itemize}
  \item the example for $q=14$ in \cite{ArnouxSchmidt09}    
  \item the example in Remark 6 and several examples in Sections 4.2
    and 4.3 in \cite{DoSchmidt16}.
  \end{itemize}
\end{corollary}

There may be other examples in
\cite{ArnouxSchmidt09,CaltaSchmidt13,DoSchmidt16} where
\Cref{theorem:criterion} applies. We only checked the examples where
the minimal polynomial of $\lambda$ was mentioned in the papers.

\Cref{theorem:SAF} gives a geometric reason for the vanishing of the
SAF invariant: an orientation-reversing symmetry.

\begin{question}
  Is the vanishing of the SAF always a consequence of some
  symmetry?
\end{question}

\subsection{Some remarks}

Examples of pseudo-Anosov homeomorphisms of nonorientable surfaces are
scarce in the literature. However, the general theory is the same as
for orientable surfaces \cite{FLP, Thurston88}. Some examples are
found in 
\begin{itemize}
\item \cite{ArnouxYoccoz81}, where the surface is the thrice punctured
  projective plane;
\item \cite{Penner88}, where the surface is the
  connected sum of two Klein bottles;
\item \Cref{sec:AY-and-AR} of this paper, where the surface is any
  closed nonorientable surface of genus at least 4.
\end{itemize}
Penner's method \cite{Penner88} is in fact general enough to construct
pseudo-Anosov mapping classes on every nonorientable surface that
allows them \cite{StrennerDegrees}.

The way in which ``nonorientable lift'' is defined may seem
unnecessarily restrictive. Why not allow branched covers or higher
degree covers of nonorientable surfaces? Any such covering would
factor through the oriented double cover, so in fact no generality is
lost.

\section{Proof of the vanishing}

\begin{lemma}\cite[Lemma 2]{CaltaSchmidt13}\label{lemma:vanishing_SAF_crit}
  A pseudo-Anosov map with stretch factor $\lambda$ has vanishing SAF
  invariant if and only if
  $\QQ(\lambda) = \QQ(\lambda + \frac{1}\lambda)$.
\end{lemma}

Note that $\QQ(\lambda) : \QQ(\lambda + \frac{1}\lambda)$ is either 1
or 2.

\begin{lemma}\label{lemma:Galois-conjugates}
  $\QQ(\lambda) : \QQ(\lambda + \frac{1}\lambda) = 2$ if and only if
  $\lambda$ and $1/\lambda$ are Galois conjugates.
\end{lemma}
\begin{proof}
  $\lambda$ and $1/\lambda$ are Galois conjugates if and only if there
  is an automorphism $\sigma$ of $\QQ(\lambda)$ such that
  $\sigma(\lambda) = 1/\lambda$. Note that $\sigma$ acts trivially on
  $\QQ(\lambda + \frac{1}\lambda)$. Such a $\sigma$ exists when
  $\QQ(\lambda) : \QQ(\lambda + \frac{1}\lambda) = 2$, but not if 
  $\QQ(\lambda) = \QQ(\lambda + \frac{1}\lambda)$.
\end{proof}

We remark that a similar lemma with a similar proof appears in
\cite{DoSchmidt16} as Proposition 1, and it is used to prove following
version of \Cref{lemma:vanishing_SAF_crit}: A pseudo-Anosov map with
stretch factor $\lambda$ has vanishing SAF invariant if and only if
the minimal polynomial of $\lambda$ is reciprocal \cite[Theorem
1]{DoSchmidt16}.

\begin{proposition}\label{prop:lambda-and-reciprocal-not-conjugates}
  Let $\lambda$ be the stretch factor of a pseudo-Anosov homeomorphism of a
  nonorientable surface. Then $\lambda$ and $1/\lambda$ are not Galois
  conjugates.
\end{proposition}
\newcommand{\tpsi}{\widetilde{\psi}}
\begin{proof}
  Denote the surface by $N$ and the pseudo-Anosov map by $\psi$. There
  is a degree 2 cover $S \to N$, where $S$ is an orientable surface,
  and $\psi$ lifts to an orientation-preserving pseudo-Anosov
  homeomorphism $\tpsi$ of $S$ whose invariant foliations $\calF^u$
  and $\calF^s$ are transversely orientable.

  It is well-known that $\lambda$ and $1/\lambda$ (or $-\lambda$ and
  $-1/\lambda$) are eigenvalues of
  $\tpsi^* : H^1(S,\RR) \to H^1(S,\RR)$, and every other eigenvalue
  $\eta$ satisfies $1/\lambda < |\eta| < \lambda$ \cite[Theorem
  5.3]{McMullen03a}. The characteristic polynomial $\chi(\tpsi^*)$ has
  integral coefficients, since $\tpsi^*$ acts on $H^1(S,\ZZ)$.

  Let $h:S\to S$ be the orientation-reversing deck transformation. We
  have $H^1(S,\RR) = W^+ \oplus W^-$, where $W^+$ and $W^-$ are the
  $\pm1$-eigenspaces of $h^*$. Since $h$ commutes with $\tpsi$, the
  subspaces $W^+$ and $W^-$ are invariant under $\tpsi^*$. In
  particular, $\chi(\tpsi^*) = \chi(\tpsi^*|W^+) \chi(\tpsi^*|W^-)$,
  and the polynomials on the right have integral coefficients.

  Note that $\calF^u$ and $\calF^s$ are represented by 1-forms
  $\omega^u,\omega^s \in H^1(S,\RR)$. We have
  $h(\calF^u) = \pm \calF^u$ and $h(\calF^s) = \pm \calF^s$, where one
  of the signs is positive, the other one is negative. Hence
  $\omega^u \in W^+$ and $\omega^s \in W^-$ or the other way around.
  In particular, $\pm \lambda$ and $\pm 1/\lambda$ are roots of
  different factors of $\chi(\tpsi^*)$.
\end{proof}

\Cref{theorem:SAF} is a corollary of
\Cref{lemma:vanishing_SAF_crit,lemma:Galois-conjugates} and
\Cref{prop:lambda-and-reciprocal-not-conjugates}.

Note that \Cref{prop:lambda-and-reciprocal-not-conjugates} is true
even if none of the invariant foliations of $\psi$ is transversely
orientable. One can see this by lifting $\psi$ to a branched double
cover orienting one of the foliations.

\section{A geometric point of view}
  
One may also desire a geometric proof of \Cref{theorem:SAF} that uses
the definition of the SAF invariant. We sketch such an argument below.

Let $\psi$ be a pseudo-Anosov homeomorphism of a nonorientable surface. Assume
that its unstable foliation $\calF^u$ is orientable. Take a one-sided
simple closed curve $c$ transverse to $\calF^u$. Normalize the measure
on $\calF^u$ so that $c$ has measure $1/2$. The flow of
$\calF^u$ induces an interval exchange map $f$ on the boundary of the
Moebius band neighborhood of $c$. Note that the intervals come in
pairs that are interchanged by $f$.

When lifted to the orientable double cover, $c$ lifts to a curve of
measure 1, and the induced interval exchange is defined by the formula
\begin{displaymath}
  \tilde{f}(x) = f(x) + 1/2\quad (\mbox{mod }1).
\end{displaymath}
It is straightforward to check that $SAF(\tilde{f}) = 0$ using the
fact that the terms in the sum come in pairs such as
$l_i \wedge (t_i \pm 1/2) + l_i \wedge (-t_i \pm 1/2)$ and this causes
cancellations.

\section{The certificate}

A monic degree $n$ polynomial $p(x) = x^n + \cdots + a_{n-1}x + a_n$ is
\emph{reciprocal} if $p(x) = x^n p(1/x)/a_n$. We will use the
following well-known result for the field with two elements.

\begin{lemma}\cite[Theorem 8.14]{McMullenNotes}\label{lemma:reciprocal}
  Let $f: V \to V$ be a linear transformation of the vector space $V$
  over a field $K$. If $f$ preserves a non-degenerate bilinear
  form, then the characteristic polynomial $\chi(f)$ is reciprocal.
\end{lemma}

\begin{proposition}\label{prop:reciprocal-mod-2}
  Let $\phi$ be a homeomorphism of the closed nonorientable surface
  $N$ of genus $g$. The characteristic polynomial $p(x)$ of
  $\phi^* : H^1(S,\ZZ) \to H^1(S,\ZZ)$ is reciprocal mod 2.
\end{proposition}
\begin{proof}
  Recall that $H_1(N,\ZZ) = \ZZ^{g-1} \oplus \ZZ_2$, so
  $H^1(N,\ZZ) \cong \ZZ^{g-1}$ and $H^1(N,\ZZ_2) \cong \ZZ_2^{g}$. The
  cup product on $H^1(N,\ZZ_2)$ is a non-degenerate bilinear form
  \cite[Example 3.8]{Hatcher02}. 

  Consider the mod 2 reduction $r : H^1(N,\ZZ) \to H^1(N,\ZZ_2)$. The
  image $\im(r) \subset H^1(N,\ZZ_2)$ has codimension 1. It is
  $\phi^*$-invariant and $\phi^*|{\im(r)}$ can be described by the
  same matrix as $\phi^*|H^1(N,\ZZ)$. Since $V$ has codimension 1, the
  characteristic polynomial of $\phi^*|{H^1(N,\ZZ_2)}$ is $(x+1)p(x)$
  or $xp(x)$ mod 2, but the latter is not possible since $\phi^*$ is
  invertible. According to \Cref{lemma:reciprocal}, $p(x)(x+1)$ is
  reciprocal mod 2 and hence the same holds for $p(x)$.
\end{proof}

\begin{proof}[Proof of \Cref{theorem:criterion}]
  Suppose $\tilde\psi$ is a lift of the pseudo-Anosov map
  $\psi: N \to N$, where $N$ is the closed genus $g+1$ nonorientable
  surface. The orientable invariant foliation of $\psi$ is represented
  by an element of $H^1(N,\RR)$ and it is an eigenvector of
  $\psi^*|H^1(N,\RR)$ with eigenvalue $\pm \lambda$ if it is the
  unstable foliation and $\pm 1/\lambda$ if it is the stable
  foliation.

  Let $p(x)\in\ZZ[x]$ be the characteristic polynomial of
  $\psi^*|H^1(N,\RR)$. Note that $\lambda$ is a root of $p(x)$,
  $p(-x)$, $x^gp(1/x)$ or $x^gp(-1/x)$. The statement of the theorem
  now follows from \Cref{prop:reciprocal-mod-2}.
\end{proof}

\section{The Arnoux--Yoccoz and Arnoux--Rauzy examples as nonorientable
  lifts}
\label{sec:AY-and-AR}

In this section we define pseudo-Anosov homeomorphisms of
nonorientable surfaces whose lifts by the orientable double cover are
the well-known pseudo-Anosov homeomorphisms of orientable surfaces
defined by Arnoux and Yoccoz \cite{ArnouxYoccoz81}. This justifies the
statement made in the introduction that the Arnoux--Yoccoz examples are
nonorientable lifts. 

Let $M$ be a Moebius band with meridian $\gamma$. Define a measured
foliation on $M$ transverse to $\gamma$ such that the length of
$\gamma$ is 1. (See \Cref{fig:AY}, where $M$ arises from identifying
the vertical sides of the rectangle by a flip, and leaves of the
foliation are vertical lines.)
\begin{figure}[ht]
  \centering
  \begin{tikzpicture}[scale=12]
    \def\rad{0.005}
    \def\colorstrength{40}
    \def\wrappingindex{2}
    \def\height{0.1}
    \def\midpointarray{0.259395/1/1/0.2594, 0.7781851/1/1/0.2594, 0.1721517/-1/2/0.1346, 0.4412948/-1/2/0.1346, 0.6456808/-1/3/0.0698, 0.7853096/-1/3/0.0698, 0.891343/-1/4/0.0362, 0.963781/-1/4/0.0362}

    \def\endpointarray{{0/1/red, 0.5187901/1/red/yellow/0/, 0.0375801/-1/red/yellow/0/3.6218987807, 0.3067233/-1/red/green/0/, 0.5758664/-1/red/green/0/, 0.7154952/-1/red/blue/0/, 0.855124/-1/red/blue/0/, 0.927562/-1/red/cyan/0/, 1/-1/red/cyan/0/}}


    \foreach \list in \endpointarray{
      \def\lastx{}
      \def\lastsign{}

      \foreach \x/\sign/\notused/\col/\isflipped/\mp[count=\i from 0,remember=\x as \lastx, remember=\sign as \lastsign] in \list{
        \ifnum \i > 0

        \def\darkcolor{\col!\colorstrength!white}
        \ifnum \isflipped = 0 
        \def\bp{\colorstrength}
        \def\ep{0}
        \else
        \def\bp{0}
        \def\ep{\colorstrength}
        \fi

        \ifnum \wrappingindex = \i

        \ifnum \sign = -1
        \fill[color=\col!\colorstrength!white]  (\lastx, 0) rectangle (1,\lastsign*\height);
        \fill[color=\col!\colorstrength!white]  (0,0) rectangle (\x, \sign*\height);
        \else
        \fill[color=\col!\colorstrength!white] (\lastx, 0) rectangle (\x, \sign*\height);

        \fi

        \else

        \fill[color=\col!\colorstrength!white] (\lastx, 0) rectangle (\x, \sign*\height);

        \fi
        \fi
      }
    }

    \foreach \list in \endpointarray{
      \foreach \x/\sign in \list{
        \draw (\x,0) -- (\x,\sign*\height);
      }
    }

    \foreach \sign in {-1,1}{
      \draw[dashed] (0,\sign*\height) -- +(1,0);
    }

    \draw[very thick] (0,0) -- (1,0);

    \foreach \list in \endpointarray{
      \foreach \x/\sign/\singcol in \list{
        \filldraw[fill=\singcol, draw=black] (\x,\sign*\height) circle (\rad);
      }
    }

    \newcommand\drawSeparatrices[1]{
      \foreach \x/\yone/\ytwo in #1{
        \draw (\x,\yone*\height) -- (\x,\ytwo*\height);
      }
    }
  
  \definecolor{separatrixcolor}{rgb}{0.0,0.392156862745,0.0}
  \tikzstyle{separatrix opts}=[dashed,very thick]
  \def\intersections{0/-1.0/1.0, 0.4812099/0/1.0, 0.5187901/0/1.0, 0.0375801/0/-1.0, 0.3067233/0/-1.0, 0.575866387689/-1.0/1.0, 0.0570763/0/1.0, 0.715495169328/-1.0/1.0, 0.1967051/0/1.0, 0.855123950966/-1.0/1.0, 0.3363339/0/1.0, 0.927561975483/-1.0/1.0, 0.4087719/0/1.0}

  \begin{scope}[color=separatrixcolor,separatrix opts]
    \drawSeparatrices{\intersections}
  \end{scope}
  
  \definecolor{curvecolor}{rgb}{1.0,0.0,0.0}
  \tikzstyle{curve opts}=[ultra thick]
  \def\intersections{0.005/0/1, 0.5237901/0/1}
  
  \begin{scope}[color=curvecolor,curve opts]
    \drawSeparatrices{\intersections}
  \end{scope}
  
  \draw[color=curvecolor,curve opts] (0.005,0.004) -- (0.523790063676,0.004);

  \path (0,\height) -- node[above] {$\alpha$} (0.518,\height); 
  \path (0.518,\height) -- node[above] {$\alpha$} (1.03,\height); 
  \path (0.03,-\height) -- node[below] {$\alpha^2$} (0.3,-\height);
  \path (0.3,-\height) -- node[below] {$\alpha^2$} (0.57,-\height);
  \path (0.57,-\height) -- node[below] {$\alpha^3$} (0.71,-\height);
  \path (0.71,-\height) -- node[below] {$\alpha^3$} (0.85,-\height);
  \path (0.85,-\height) -- node[below] {$\alpha^4$} (0.92,-\height);
  \path (0.92,-\height) -- node[below] {$\alpha^4$} (1,-\height);
 
  \node[right] at (1,0) {$\gamma$};
  \node[above] at (0.27,0) {$\gamma'$};

  \end{tikzpicture}
  \caption{The nonorientable surface.}
  \label{fig:AY}
\end{figure}
Fix some $g \ge 3$, and divide the boundary of $M$ into intervals of
lengths $\alpha,\alpha,\alpha^2,\alpha^2,\ldots,\alpha^g,\alpha^g$,
where $\alpha$ is the unique root of the polynomial
$x^g+x^{g-1}+\ldots+x-1$ lying in the interval $[0,1]$. Identification
of pairs of intervals of the same length by translations gives rise to
a singular measured foliation $\calF$ of $N_g$, the closed
nonorientable surface of genus $g+1$ whose orientable double cover is
$S_g$, the closed orientable surface of genus $g$. The curve $\gamma$
is one-sided and is transverse to $\calF$.

Consider another transverse curve $\gamma'$ that is obtained by a small
perturbation of the U-shaped curve shown on \Cref{fig:AY}. By following
the leaves of $\calF$ emanating from the unique singularity until they
hit $\gamma'$, one can see that the intervals appearing along
$\gamma'$ have lengths
$\alpha^2,\alpha^2,\alpha^3,\alpha^3,\ldots,
\alpha^{g+1},\alpha^{g+1}$,
in this order. It follows that there is a homeomorphism of $\gamma$ to
$\gamma'$ that brings intervals to intervals by shrinking by a factor
of $\alpha$, and this map extends to the homeomorphism $\psi$ of the
surface such that $\psi (\calF) = \alpha \calF$. 

We claim that the lift of $\psi$ to the orientable double cover is the
Arnoux--Yoccoz example on $S_g$. First notice that $\calF$ is
transversely orientable, hence so is the lift $\tilde\calF$. The curve
$\gamma$ lifts to a two-sided curve $\tilde\gamma$ transverse to
$\tilde\calF$ whose measure is 2. The induced interval exchange
transformation is the transformation $T$ in \cite{ArnouxYoccoz81} (up
to scaling by 2). See \Cref{fig:lift}.

\begin{figure}
  \centering
  \begin{tikzpicture}[scale = 12]
\def\rad{0.005}
\def\colorstrength{40}
\def\wrappingindex{8}
\def\height{0.1}

\def\midpointarray{0.1296975/1/1/0.2594/$\alpha$, 0.3890925/1/2/0.2594/$\alpha$, 0.5860758/1/3/0.1346/$\alpha^2$, 0.7206474/1/4/0.1346/$\alpha^2$, 0.8228404/1/5/0.0698/$\alpha^3$, 0.8926548/1/6/0.0698/$\alpha^3$, 0.9456715/1/7/0.0362/$\alpha^4$, 0.9818905/1/8/0.0362/$\alpha^4$, 0.0860758/-1/4/0.1346/$\alpha^2$, 0.2206474/-1/3/0.1346/$\alpha^2$, 0.3228404/-1/6/0.0698/$\alpha^3$, 0.3926548/-1/5/0.0698/$\alpha^3$, 0.4456715/-1/8/0.0362/$\alpha^4$, 0.4818905/-1/7/0.0362/$\alpha^4$, 0.6296975/-1/2/0.2594/$\alpha$, 0.8890925/-1/1/0.2594/$\alpha$}

\def\endpointarray{{0/1/red, 0.259395/1/green/orange/0/, 0.5187901/1/red/green/0/, 0.6533616/1/green/blue/0/, 0.7879332/1/red/cyan/0/, 0.8577476/1/green/magenta/0/, 0.927562/1/red/yellow/0/, 0.963781/1/green/gray/0/, 1/1/red/brown/0/}, {0.0187901/-1/green/red/0/3.62190772036, 0.1533616/-1/red/cyan/0/, 0.2879332/-1/green/blue/0/, 0.3577476/-1/red/yellow/0/, 0.427562/-1/green/magenta/0/, 0.463781/-1/red/brown/0/, 0.5/-1/green/gray/0/, 0.759395/-1/red/green/0/, 0.0187901/-1/green/orange/0/3.62190772036}}


\foreach \list in \endpointarray{
\def\lastx{}
\def\lastsign{}

\foreach \x/\sign/\notused/\col/\isflipped/\mp[count=\i from 0,remember=\x as \lastx, remember=\sign as \lastsign] in \list{
\ifnum \i > 0

\def\darkcolor{\col!\colorstrength!white}
\ifnum \isflipped = 0 
\def\bp{\colorstrength}
\def\ep{0}
\else
\def\bp{0}
\def\ep{\colorstrength}
\fi

\ifnum \wrappingindex = \i

\ifnum \sign = -1
\fill[left color=\darkcolor, right color=\darkcolor]  (\lastx, 0) rectangle (1,\lastsign*\height);
\fill[left color=\darkcolor, right color=\darkcolor]  (0,0) rectangle (\x, \sign*\height);
\else
\fill[left color=\darkcolor, right color = \darkcolor] (\lastx, 0) rectangle (\x, \sign*\height);
\fi

\else

\fill[left color=\darkcolor, right color = \darkcolor] (\lastx, 0) rectangle (\x, \sign*\height);

\fi
\fi
}
}

\newcommand\definepos[1]{
\ifnum #1 = 1
\def\pos{above}
\else
\def\pos{below}
\fi
}

\foreach \list in \endpointarray{
\foreach \x/\sign in \list{
\draw (\x,0) -- (\x,\sign*\height);
}
}

\begin{scope}[red, ultra thick]
  \draw (0.005,\height+0.005) -- ++(0,-\height)  -- ++(0.259395,0) -- ++(0,\height);
  \draw (0.505,-\height-0.005) -- ++(0,\height) -- ++(0.259395,0) -- ++(0,-\height);  
\end{scope}

\foreach \sign in {-1,1}{
\draw[dashed] (0,\sign*\height) -- +(1,0);
}

\draw[very thick] (0,0) -- (1,0);

\foreach \list in \endpointarray{
\foreach \x/\sign/\singcol in \list{
\filldraw[fill=\singcol, draw=black] (\x,\sign*\height) circle (\rad);
}
}

\foreach \x/\sign/\label/\length/\text in \midpointarray{
\definepos{\sign}
\node at (\x,\sign*\height) [\pos] {\text};
}

\node[right] at (1,0) {$\tilde\gamma$};
\node[above] at (0.1296975, 0) {$\tilde\gamma'$};

\end{tikzpicture}

  \caption{The oriented double cover.}
  \label{fig:lift}
\end{figure}
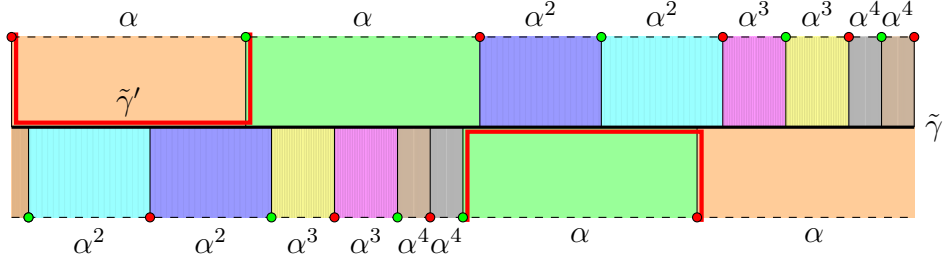

The lift $\tilde \psi$ maps $\tilde \gamma$ to $\tilde \gamma'$, so
the interval exchanges $f$ and $f'$ induced by the two curves are
conjugates by scaling by $\alpha$. Note also that $f'$ is the same as
the restriction of $f$ to the initial subinterval of length $2\alpha$.
Arnoux and Yoccoz use exactly this restriction to construct their
homeomorphism. Hence their example coincides with $\tilde\psi$.

Before Arnoux and Yoccoz gives the above mentioned examples in their
paper, they produce a pseudo-Anosov homeomorphism of the thrice
punctured projective plane. The genus 3 example in the Arnoux--Yoccoz
family is actually a lift of this map. Arnoux and Rauzy has given many
more examples on the thrice punctured projective plane. The lifts of
these examples to $S_3$ are the Arnoux--Rauzy examples mentioned in
\cite[Section 4.1]{DoSchmidt16}, hence they are also nonorientable
lifts.

\section{Acknowledgements} 

The paper was written while the author was a member at the Institute
for Advanced Study. The author thanks the IAS for its hospitality. He
also thanks the referee for thoughtful remarks and Livio Liechti for
correcting some typos.

This work was supported by NSF grant DMS-1128155.

\bibliographystyle{alpha}
\bibliography{../mybibfile}

\end{document}